\documentclass[10pt,a4paper]{amsart}

\usepackage[T1]{fontenc}
\usepackage[utf8]{inputenc}
\usepackage{lmodern}
\usepackage{microtype}
\usepackage{amsmath,amssymb,amsthm,mathtools}
\usepackage{enumitem}
\usepackage{xcolor}
\usepackage[colorlinks=true,linkcolor=blue,citecolor=blue,urlcolor=blue]{hyperref}
\hypersetup{pdftitle={Slice and Partition Rank Criteria for Polynomial Zero-Avoidance},pdfauthor={Simone Costa, Stefano Della Fiore, and Mattia Fontana}}
\usepackage[margin=1.15in]{geometry}

\newtheorem{theorem}{Theorem}[section]
\newtheorem{proposition}[theorem]{Proposition}
\newtheorem{corollary}[theorem]{Corollary}
\newtheorem{lemma}[theorem]{Lemma}

\theoremstyle{definition}
\newtheorem{definition}[theorem]{Definition}
\newtheorem{remark}[theorem]{Remark}
\newtheorem{example}[theorem]{Example}

\newcommand{\F}{\mathbb F}

\newcommand{\supp}{\operatorname{supp}}
\newcommand{\srk}{\operatorname{srk}}
\newcommand{\prk}{\operatorname{prk}}

\newcommand{\EGZ}{\operatorname{EGZ}}

\newcommand{\boxt}{\boxtimes}
\newcommand{\1}{\mathbf 1}

\newcommand{\Tau}{\mathcal T}
\newcommand{\StirlingTwo}[2]{\genfrac\{\}{0pt}{}{#1}{#2}}
\newcommand{\StirlingOne}[2]{\genfrac[]{0pt}{}{#1}{#2}}

\title[Slice and partition rank criteria]
{Slice and Partition Rank Criteria for Polynomial Zero-Avoidance}
\author{Simone Costa}
\address{DICATAM, Sezione di Matematica, Universit\`a degli Studi di Brescia, Via Branze 43, 25123 Brescia, Italy}
\email{simone.costa@unibs.it}
\author{Stefano Della Fiore}
\address{DII, Universit\`a degli Studi di Brescia, Via Branze 38, 25123 Brescia, Italy}
\email{stefano.dellafiore@unibs.it}
\author{Mattia Fontana}
\address{Department of Engineering and Sciences, Universitas Mercatorum, Rome, Italy}
\email{mattia.fontana@studenti.unimercatorum.it}

\begin{document}
\begin{abstract}
We study polynomial zero-avoidance over finite vector spaces by means of slice rank and partition rank.  We first make the support-entropy method effective by showing how a finite dual certificate yields an explicit entropy gap whenever the coefficient support admits no probability distribution with uniform marginals.  For the quadratic elementary symmetric polynomial over fields of characteristic three, the ternary structure of the coefficient support gives a certificate with optimal normalized margin and a uniform analytic bound for the corresponding higher-degree Erd\H{o}s--Ginzburg--Ziv constant, avoiding a separate optimization for each field.

We then use partition rank to handle solutions in pairwise distinct variables.  Equality profiles are encoded by contracted local tensors, reducing the global problem to finitely many slice-rank estimates.  Applying this reduction on the multiplicative torus gives restricted-alphabet zero-sum bounds with exponential base below the alphabet size.  Coordinatewise inversion and support stratification then yield, to the best of our knowledge, the first nontrivial exponential bound for the higher-degree Erd\H{o}s--Ginzburg--Ziv problem over $\F_5^n$ associated with the fourth elementary symmetric polynomial.
\end{abstract}

\maketitle

\section{Introduction}

Zero-sum theory studies additive configurations in finite abelian groups.  A classical example is the Erd\H{o}s--Ginzburg--Ziv constant, which asks how long a sequence must be before it contains a zero-sum subsequence of prescribed length.  Caro and Schmitt~\cite{CaroSchmitt} introduced a higher-degree version by replacing the sum with an elementary symmetric polynomial.  For a finite commutative ring $R$, they denote by $\EGZ(t,R,m)$ the least integer $\ell$ such that every sequence of length at least $\ell$ in $R$ contains a subsequence of length $t$ on which $e_m$ vanishes.

This paper considers these questions over vector spaces $\F_q^n$.  The main tools are slice rank and partition rank.  Both methods start from a tensor detecting the required polynomial equation.  After restriction to the set under study, a diagonal tensor has rank equal to the size of the set, and an upper bound follows from a suitable estimate for the ambient tensor.

Two points require additional care.  The first concerns slice rank.  In fixed product bases, the support-entropy bound is strictly smaller than the trivial alphabet bound only when the coefficient support carries no probability distribution with uniform marginals.  This condition is easy to state, but by itself gives no numerical gap.  We formulate it as a finite linear feasibility problem and use a separating hyperplane to obtain an explicit dual certificate.  Combining the certificate with an information-theoretic estimate gives a quantitative entropy gap that can be checked directly.

We apply this idea to
\[
 e_2(x,y,z)=xy+yz+zx
\]
over $\F_{3^k}$.  The coefficient support has a simple description in terms of ternary digits, which gives a certificate with optimal normalized margin and the one-variable analytic base
\[
 C_k=\inf_{x\ge1}\left\{1+x^{-4k/3}\bigl((1+x+x^2)^k-1\bigr)\right\}<3^k.
\]
Consequently,
\[
 \EGZ(3,\F_{3^k}^n,2)\le C_k^{\,n+o(n)}\qquad(k>1).
\]
An earlier arXiv version of a paper by the first two authors~\cite{CostaDellaFiore} contained the corresponding qualitative support-entropy gap, but that general result was not included in the published article.  We recall it in Proposition~\ref{prop:q3-support-gap}. 

The second difficulty is caused by pairwise distinctness.  Direct slice-rank extraction does not remove partial diagonals, that is, tuples in which only some variables coincide.  Following Naslund~\cite{NaslundPartition}, we expand a distinctness indicator over the lattice of set partitions.  Each equality profile gives a contracted polynomial and a tensor detecting its coordinatewise zeros, reducing the original problem to finitely many slice-rank estimates.  Omar's partition-lattice formalism~\cite{OmarPartition} provides the natural M\"obius-inversion language for this reduction.

More precisely, contractions with many blocks are controlled by degree, {while selected low-block contractions can be removed when their zero sets consist only of diagonal points.}  {Although partition rank performs the global distinct-variable reduction, the remaining estimates are slice ranks of contracted local tensors.}  Thus the slice-rank argument depends on the coefficient support, while the partition-rank argument depends on the tensors obtained from the contractions of the polynomial.

We use the partition-rank reduction for the elementary symmetric polynomials
\[
 e_m(x_1,\ldots,x_t)=\sum_{1\le i_1<\cdots<i_m\le t}x_{i_1}\cdots x_{i_m}.
\]
A set $A\subseteq\F_q^n$ is called \emph{strongly $e_m$-free of length $t$} if every coordinatewise zero of $e_m$ in $A^t$ is constant.  This is stronger than the condition defining $\EGZ(t,\F_q^n,m)$, since repeated values are allowed in the strong problem.  The elementary degree argument gives a nontrivial estimate when $t>2m$; the partition-rank method is designed for the distinct-variable problem and also applies when this condition fails.

A useful feature of the method is that it works on a restricted alphabet.  We apply it to the multiplicative torus $(\F_q^\times)^n$.  For the linear equation, this gives a zero-sum bound whose exponential base is strictly smaller than $q-1$, the size of the torus alphabet.  In characteristic three this becomes a restricted-alphabet cap-set problem; for $q=9$ the resulting base improves the one obtained by viewing $\F_9^n$ as $\F_3^{2n}$ and applying the ordinary cap-set estimate.  The same toric argument for $q=5$ improves, on the restricted alphabet, the base inherited from Sauermann's full-space theorem~\cite{SauermannDistinct}.

Coordinatewise inversion transfers the toric linear bounds to $e_{p-1}$.  Stratifying the full space according to vector supports then allows the toric estimate to be applied on every nonzero stratum.  Put
$$
 \vartheta_q=\inf_{0<\rho<1}(1+\rho+\cdots+\rho^{q-2})\rho^{-(q-1)/3},
 \qquad
 \eta_{q,p}=\inf_{0<\rho<1}(1+\rho+\cdots+\rho^{q-2})\rho^{-(q-1)/p},
$$
and set $
 B_p=\sum_{r=3}^p r\StirlingTwo{p}{r}.
$
For $q=p^k>4$ we prove
\[
 \EGZ(p,\F_q^n,p-1)\le(p-1)B_p(1+\vartheta_q)^n+1,
 \qquad1+\vartheta_q<q.
\]
The main new case is the following.
\begin{theorem}[The case $q=p=5$]\label{thm:intro-q5}
One has
\[
 \EGZ(5,\F_5^n,4)\le480(1+\vartheta_5)^n+1,
 \qquad1+\vartheta_5\approx4.9556902<5.
\]
Moreover, every strongly $e_4$-free subset of $\F_5^n$ of length five has cardinality at most
\[
 5(1+\eta_{5,5})^n,
 \qquad1+\eta_{5,5}\approx4.2563617<5.
\]
\end{theorem}

The paper is organized as follows.  Section~\ref{sec:slice-criteria} recalls the slice-rank tools and introduces the dual certificate.  Section~\ref{sec:q3-gap} treats the quadratic problem in characteristic three.  Section~\ref{sec:partition-criteria} develops the partition-rank reduction through equality profiles.  Section~\ref{sec:zero-sum} applies the method to restricted-alphabet and toric zero-sum problems.  Finally, Section~\ref{sec:higher-egz} transfers the toric estimates to higher-degree EGZ constants.

\section{Slice-rank criteria and dual certificates}\label{sec:slice-criteria}

We shall use two standard slice-rank arguments.  The low-degree method bounds a diagonal tensor detecting a polynomial zero set by counting low-degree monomials.  The support-entropy method instead uses the coefficient support in a fixed product basis.  {We recall both arguments mainly to fix notation, record the quantitative dependence and keep the characteristic-three application self-contained~\cite{CrootLevPach,EllenbergGijswijt,TaoBlog,TaoSawin}.}

\subsection{Diagonal extraction and support entropy}
Throughout this subsection we assume $r\ge2$.  Let $X_1,\ldots,X_r$ be finite sets and let $\F$ be a field.  A tensor is simply a function
\[
 T:X_1\times\cdots\times X_r\longrightarrow\F.
\]

We use Tao's slice-rank terminology~\cite{TaoBlog,TaoSawin}.

\begin{definition}
A tensor $T$ is a \emph{slice} if, for some $i\in[r]$, it can be written as
\[
 T(x_1,\ldots,x_r)=f(x_i)g(x_1,\ldots,x_{i-1},x_{i+1},\ldots,x_r).
\]
The \emph{slice rank} $\srk(T)$ is the least number of slices whose sum is $T$.
\end{definition}

We shall repeatedly use Tao's diagonal-tensor lemma~\cite[Lemma~1]{TaoBlog}.

\begin{lemma}[Diagonal tensor lemma]\label{lem:diagonal}
Let $A$ be a finite set and let
\[
 D(x_1,\ldots,x_r)=
 \begin{cases}
 c_a,&x_1=\cdots=x_r=a\in A,\\
 0,&\text{otherwise},
 \end{cases}
\]
where every $c_a$ is nonzero.  Then
\[
 \srk(D)=|A|.
\]
\end{lemma}

{Restriction to smaller domains cannot increase slice rank.}  We shall also use the coordinatewise tensor power
\[
 T^{\boxt n}(x^{(1)},\ldots,x^{(r)})
 :=\prod_{j=1}^nT(x^{(1)}_j,\ldots,x^{(r)}_j),
\]
where $x^{(i)}\in X_i^n$.

The next theorem is a finite-type version, with an explicit polynomial prefactor, of the tensor-power support bound of Tao and Sawin~\cite[Proposition~6]{TaoSawin}.  We include the short type-decomposition proof because this precise normalization is used repeatedly below.

Let $X$ be a finite set of cardinality $q$, and choose, for each $i\in[r]$, a basis
\[
 \mathcal B_i=(b_{i,0},\ldots,b_{i,q-1})
\]
of the $q$-dimensional space $\F^X$.  Write $\mathcal B=(\mathcal B_1,\ldots,\mathcal B_r)$ for the resulting product basis and write
\[
 T(x_1,\ldots,x_r)
 =\sum_{\alpha\in\{0,\ldots,q-1\}^r}
 c_\alpha\prod_{i=1}^r b_{i,\alpha_i}(x_i)
\]
and define the support
\[
 \Gamma_{\mathcal B}(T)=\{\alpha:c_\alpha\ne0\}.
\]
The support depends on the chosen bases.

For a probability distribution $\mu$ on a finite set, let
\[
 H(\mu)=-\sum_x\mu(x)\log\mu(x)
\]
with the convention $0\log0=0$.  For a nonempty set $\Gamma\subseteq\{0,\ldots,q-1\}^r$, let $\mathcal P(\Gamma)$ denote the set of probability distributions on $\Gamma$.  For $\mu\in\mathcal P(\Gamma)$, denote its $i$-th marginal by $\mu_i$ and set
\[
 H(\Gamma):=
 \max_{\mu\in\mathcal P(\Gamma)}\min_{1\le i\le r}H(\mu_i).
\]

\begin{theorem}[Finite-type support-entropy bound]\label{thm:support-entropy}
Let $T:X^r\to\F$ be a nonzero tensor, and let $\Gamma=\Gamma_{\mathcal B}(T)$ in fixed product bases.  Then
\[
 \srk(T^{\boxt n})
 \le (n+1)^{|\Gamma|}\exp\bigl(nH(\Gamma)\bigr).
\]
Consequently, if $A\subseteq X^n$ and the restriction of $T^{\boxt n}$ to $A^r$ is diagonal with nonzero diagonal, then
\[
 |A|\le (n+1)^{|\Gamma|}\exp\bigl(nH(\Gamma)\bigr).
\]
\end{theorem}

\begin{proof}
Expand $T^{\boxt n}$ using the chosen bases.  Its summands are indexed by words
\[
 \boldsymbol\alpha=(\alpha^{(1)},\ldots,\alpha^{(n)})\in\Gamma^n.
\]
For such a word, let $\nu_{\boldsymbol\alpha}$ be its empirical distribution on $\Gamma$.  There are at most $(n+1)^{|\Gamma|}$ possible empirical distributions.

Fix one such distribution $\nu$, and let $T_\nu$ be the sum of the terms having empirical distribution $\nu$.  Choose an index $i\in[r]$ satisfying
\[
 H(\nu_i)=\min_{1\le j\le r}H(\nu_j),
\]
where $\nu_i$ is the $i$-th marginal of $\nu$.  For each word
$a=(a_1,\ldots,a_n)\in\{0,\ldots,q-1\}^n$ of type $\nu_i$, group together all terms in $T_\nu$ whose $i$-th coordinate word is $a$.  This writes $T_\nu$ as a sum of slices, one for each such word $a$.

The number of words of type $\nu_i$ is the corresponding multinomial coefficient and is at most
\[
 \exp\bigl(nH(\nu_i)\bigr)
 \le \exp\bigl(nH(\Gamma)\bigr).
\]
Thus
\[
 \srk(T_\nu)\le\exp\bigl(nH(\Gamma)\bigr).
\]
Summing over the at most $(n+1)^{|\Gamma|}$ types proves the first assertion.  The second follows by restriction and Lemma~\ref{lem:diagonal}.
\end{proof}
The following extraction statement is an immediate consequence of the Tao--Sawin bound.
\begin{corollary}[Subexponential diagonal extraction]\label{cor:subexp-extraction}
Let $T:X^r\to\F$ be a fixed nonzero tensor with support $\Gamma$ in fixed product bases.  Let $\mathcal A_n$ be a finite combinatorial object endowed with a size $|\mathcal A_n|$.  Suppose that one can associate with it a set $A_n\subseteq X^n$ such that
\[
 |A_n|\ge \frac{|\mathcal A_n|}{M_n},
 \qquad \log M_n=o(n),
\]
and such that $T^{\boxt n}|_{A_n^r}$ is diagonal with nonzero diagonal.  Then
\[
 |\mathcal A_n|
 \le M_n(n+1)^{|\Gamma|}\exp\bigl(nH(\Gamma)\bigr)
 =\exp\bigl((H(\Gamma)+o(1))n\bigr).
\]
In particular, if $H(\Gamma)<\log q$, then $|\mathcal A_n|\le c^{n+o(n)}$ for some $c<q$.
\end{corollary}

\begin{proof}
This is Theorem~\ref{thm:support-entropy} multiplied by the loss $M_n$.  Since $\Gamma$ is fixed and $\log M_n=o(n)$, both $M_n$ and $(n+1)^{|\Gamma|}$ are absorbed into the $o(n)$ term.
\end{proof}

The criterion for strict improvement over the trivial value $q^n$ is the following elementary consequence of the Tao--Sawin entropy functional.  In the symmetric three-variable setting, this uniform-marginal obstruction was already used in the arXiv proof of~\cite[Theorem~4.10]{CostaDellaFiore}.

\begin{proposition}[Uniform-marginal criterion]\label{prop:uniform-marginals}
For every nonempty $\Gamma\subseteq\{0,\ldots,q-1\}^r$, one has
\[
 H(\Gamma)=\log q
\]
if and only if there exists a probability distribution on $\Gamma$ all of whose marginals are uniform on $\{0,\ldots,q-1\}$.
\end{proposition}

\begin{proof}
Every probability distribution on a set of size $q$ has entropy at most $\log q$, with equality exactly for the uniform distribution.  Since $\mathcal P(\Gamma)$ is compact and the entropy is continuous, the maximum defining $H(\Gamma)$ is attained.  Therefore $H(\Gamma)=\log q$ exactly when some maximizing distribution has every marginal entropy equal to $\log q$, which is equivalent to all marginals being uniform.
\end{proof}

\subsection{Quantitative dual certificates}

The uniform-marginal condition is a finite linear-feasibility problem.  The next statement packages two standard facts, strict hyperplane separation and Pinsker's inequality, into a quantitative certificate suited to tensor supports; see, for example,~\cite{CoverThomas} for the information-theoretic inequality.  This finite certificate formulation is included because it is optimized explicitly in Section~\ref{sec:q3-gap}.

\begin{proposition}[Dual certificate]\label{prop:dual}
Let $\Gamma\subseteq\{0,\ldots,q-1\}^r$ be nonempty, and assume that it carries no probability distribution with all marginals uniform.  Then there exist functions
\[
 w_i:\{0,\ldots,q-1\}\longrightarrow\mathbb R\qquad(i\in[r])
\]
and a number $\eta>0$ such that
\[
 \sum_{a=0}^{q-1}w_i(a)=0
\]
for every $i$, and
\[
 \sum_{i=1}^r w_i(\alpha_i)\ge\eta
 \qquad\text{for every }\alpha\in\Gamma.
\]
If $W=\sum_i\lVert w_i\rVert_\infty$, then
\[
 H(\Gamma)\le\log q-\frac{\eta^2}{2W^2}.
\]
Consequently, if $T$ is a tensor whose support in fixed product bases is $\Gamma$, then
\[
 \srk(T^{\boxt n})
 \le (n+1)^{|\Gamma|}
 \left(q\exp\left(-\frac{\eta^2}{2W^2}\right)\right)^n.
\]
\end{proposition}

\begin{proof}
Associate to $\alpha\in\Gamma$ its incidence vector
\[
 v(\alpha)=(e_{\alpha_1},\ldots,e_{\alpha_r})\in\mathbb R^{rq},
\]
where $e_a$ denotes the $a$-th standard basis vector of $\mathbb R^q$, and let $u$ be the vector whose entries in each of the $r$ blocks are all $1/q$.  A distribution on $\Gamma$ has uniform marginals exactly when $u$ belongs to the convex hull of the vectors $v(\alpha)$.  Since this does not happen, strict hyperplane separation gives functions $w_i$ and $\eta>0$ such that, after subtracting the average of each $w_i$,
\[
 \sum_iw_i(\alpha_i)\ge\eta
\]
for all $\alpha\in\Gamma$, while each $w_i$ has zero average under the uniform distribution.

Let $\mu$ be any probability distribution on $\Gamma$, with marginals $\mu_i$, and let $u_q$ be uniform on $\{0,\ldots,q-1\}$.  Taking expectations gives
\[
 \eta\le\sum_i\left|\mathbb E_{\mu_i}w_i-\mathbb E_{u_q}w_i\right|
 \le W\max_i\lVert\mu_i-u_q\rVert_1.
\]
Hence some $i$ satisfies $\lVert\mu_i-u_q\rVert_1\ge\eta/W$.  Write
\[
 D(\nu\Vert u_q):=\sum_{a=0}^{q-1}\nu(a)\log\frac{\nu(a)}{u_q(a)}
\]
for the relative entropy.  By Pinsker's inequality~\cite{CoverThomas}, with natural logarithms,
\[
 \log q-H(\mu_i)=D(\mu_i\Vert u_q)
 \ge\frac12\lVert\mu_i-u_q\rVert_1^2
 \ge\frac{\eta^2}{2W^2}.
\]
Maximizing $\min_iH(\mu_i)$ over $\mu$ proves the entropy estimate, and Theorem~\ref{thm:support-entropy} gives the final bound.
\end{proof}

\subsection{Low-degree bounds and direct strong-avoidance consequences}
We now recall the usual low-degree slice-rank bound and apply it to elementary symmetric polynomials. The degree-splitting lemma is the usual slice-rank formulation of the Croot--Lev--Pach--Ellenberg--Gijswijt argument~\cite{CrootLevPach,EllenbergGijswijt,TaoBlog}.  The strong-avoidance statements below are immediate applications, included as comparison results rather than as new slice-rank theorems.

For an integer $Q\ge2$ and $D\ge0$, define
\[
 M_Q(n,D)=
 \left|\left\{(a_1,\ldots,a_n)\in\{0,\ldots,Q-1\}^n:
 a_1+\cdots+a_n\le D\right\}\right|.
\]
We use this notation with $Q=q$ on the full field alphabet and with $Q=q-1$ on the multiplicative torus.

\begin{lemma}[Low-degree slice bound]\label{lem:degree-slice}
Let $P$ be a polynomial function on $(\F_q^n)^r$ represented by a reduced polynomial of total degree at most $D$.  Then
\[
 \srk(P)\le rM_q\left(n,\left\lfloor\frac Dr\right\rfloor\right).
\]
\end{lemma}

\begin{proof}
Expand $P$ in reduced monomials, so every scalar exponent lies in $\{0,\ldots,q-1\}$.  For each monomial, split its total degree among the $r$ vector-variable groups.  At least one group has degree at most $D/r$; assign the monomial to the least such group.  For a fixed group $i$, collect all monomials having the same monomial factor in the variables of that group.  Each collection is a slice with distinguished variable group $i$.  The number of possible distinguished monomials is at most $M_q(n,\lfloor D/r\rfloor)$.  Summing over $i$ proves the claim.
\end{proof}

For $0\le\alpha\le q-1$, set
\[
 \kappa_q(\alpha)
 :=\inf_{0<\rho<1}
 (1+\rho+\cdots+\rho^{q-1})\rho^{-\alpha}.
\]
The generating-function estimate
\[
 M_q(n,\alpha n)\le\kappa_q(\alpha)^n
\]
is immediate: for $0<\rho<1$,
\[
 M_q(n,\alpha n)\rho^{\alpha n}
 \le(1+\rho+\cdots+\rho^{q-1})^n.
\]
Moreover,
\[
 \kappa_q(\alpha)<q\qquad\text{whenever }\alpha<(q-1)/2.
\]
Indeed, with $\rho=e^s$, the derivative at $s=0$ of
\[
 \log(1+e^s+\cdots+e^{(q-1)s})-\alpha s
\]
is $(q-1)/2-\alpha>0$, so the expression is smaller than $\log q$ for negative $s$ sufficiently close to zero.

Let $F_1,\ldots,F_s\in\F_q[x_1,\ldots,x_t]$ have total degrees $d_1,\ldots,d_s$.  Write $d=d_1+\cdots+d_s$.

\begin{definition}
A set $A\subseteq\F_q^n$ is \emph{strongly $(F_1,\ldots,F_s)$-free} if the coordinatewise system
\[
 F_j(x^{(1)},\ldots,x^{(t)})=0\qquad(j\in[s])
\]
with $x^{(1)},\ldots,x^{(t)}\in A$ implies
\[
 x^{(1)}=\cdots=x^{(t)}.
\]
\end{definition}

The standard diagonal construction gives the following direct consequence.

\begin{proposition}[Strong polynomial-system consequence]\label{thm:strong-system}
Assume that
\[
 F_j(a,\ldots,a)=0
\qquad\text{for every }a\in\F_q\text{ and }j\in[s].
\]
If $A\subseteq\F_q^n$ is strongly $(F_1,\ldots,F_s)$-free, then
\[
 |A|\le
 tM_q\left(n,\left\lfloor\frac{d(q-1)n}{t}\right\rfloor\right).
\]
In particular, if $t>2d$, then
\[
 |A|\le t\,\kappa_q\left(\frac{d(q-1)}t\right)^n
\]
with
\[
 \kappa_q\left(\frac{d(q-1)}t\right)<q.
\]
\end{proposition}

\begin{proof}
For scalar variables define
\[
 \tau(x_1,\ldots,x_t)=
 \prod_{j=1}^s\bigl(1-F_j(x_1,\ldots,x_t)^{q-1}\bigr).
\]
This is the indicator of the simultaneous zero set of the $F_j$.  Therefore $\tau^{\boxt n}$ is diagonal with diagonal value $1$ on $A^t$.  By Lemma~\ref{lem:diagonal},
\[
 |A|\le\srk(\tau^{\boxt n}).
\]
The polynomial $\tau$ has total degree at most $d(q-1)$.  After reducing exponents modulo the relations $x^q=x$, its coordinatewise tensor power is represented by a reduced polynomial of total degree at most $d(q-1)n$.  Lemma~\ref{lem:degree-slice} gives the first bound.  The exponential estimate and its strictness under $t>2d$ follow from the preceding discussion.
\end{proof}

Specializing the preceding standard consequence to one elementary symmetric polynomial gives the following immediate corollary.

\begin{definition}
A set $A\subseteq\F_q^n$ is \emph{strongly $e_m$-free of length $t$} if
\[
 e_m(x^{(1)},\ldots,x^{(t)})=0
\]
coordinatewise for elements of $A$ implies that all $x^{(i)}$ are equal.
\end{definition}

\begin{corollary}[Strongly $e_m$-free consequence]\label{thm:strong-em}
Let $q$ have characteristic $p$.  Suppose
\[
 p\mid\binom tm.
\]
If $A\subseteq\F_q^n$ is strongly $e_m$-free of length $t$, then
\[
 |A|\le
 tM_q\left(n,\left\lfloor\frac{m(q-1)n}{t}\right\rfloor\right).
\]
If additionally $t>2m$, then
\[
 |A|\le t\,\kappa_{q,t,m}^{\,n},
\qquad
 \kappa_{q,t,m}:=
 \inf_{0<\rho<1}
 (1+\rho+\cdots+\rho^{q-1})
 \rho^{-m(q-1)/t}<q.
\]
\end{corollary}

\begin{proof}
The polynomial $e_m$ has total degree $m$, and
\[
 e_m(a,\ldots,a)=\binom tm a^m=0
\]
in characteristic $p$.  Apply Proposition~\ref{thm:strong-system} with one polynomial.
\end{proof}

\begin{example}
For $q=5^k$, $t=5$, and $m=2$, every strongly $e_2$-free set satisfies
\[
 |A|\le5\left[
 \inf_{0<\rho<1}(1+\rho+\cdots+\rho^{q-1})
 \rho^{-2(q-1)/5}
 \right]^n,
\]
and the base of the exponential is strictly smaller than $q$.
\end{example}

\begin{example}
For $q=7^k$ and $t=7$, the same conclusion holds for $m=2$ and for $m=3$, since
\[
 7\mid\binom72,\qquad7\mid\binom73,
\]
and $7>2m$ in both cases.
\end{example}

\begin{remark}
Strong freeness is substantially stronger than the condition used in the definition of the higher-degree EGZ constant.  That definition excludes only solutions obtained from distinct indices of a sequence; it does not exclude solutions with repeated values.  The partition-rank criteria of Section~\ref{sec:partition-criteria} and the applications of Section~\ref{sec:higher-egz} address this distinction and culminate in a new full-space higher-degree EGZ application.
\end{remark}

\section{The characteristic-three quadratic problem}\label{sec:q3-gap}

We now consider $e_2(x,y,z)=xy+yz+zx$ over fields of characteristic three.  The elementary degree condition of Corollary~\ref{thm:strong-em} does not apply, but the coefficient support still has entropy below $\log q$.  We first recall the qualitative argument from the earlier arXiv version of a paper by the first two authors~\cite[Theorem~4.10 and Remark~4.11]{CostaDellaFiore}, which was omitted from the published article.  We then describe the support in ternary digits and obtain the new certificate with optimal normalized margin and a uniform analytic bound.

\subsection{Qualitative support gap}

Let $q=3^k$ and define
\[
 \tau_q(x,y,z)=1-(xy+yz+zx)^{q-1}.
\]
This is the indicator of the equation $xy+yz+zx=0$.  In the monomial basis,
\[
 \tau_q(x,y,z)
 =1-\sum_{a+b+c=q-1}
 \binom{q-1}{a,b,c}
 x^{a+c}y^{a+b}z^{b+c}.
\]
Let $\Gamma_q$ be its support.

We shall use the following multinomial form of Lucas' theorem~\cite{Lucas1878}.  If
\[
 N=\sum_{j\ge0}N_j3^j,\qquad
 a=\sum_{j\ge0}a_j3^j,\qquad
 b=\sum_{j\ge0}b_j3^j,\qquad
 c=\sum_{j\ge0}c_j3^j
\]
are base-three expansions of nonnegative integers with $a+b+c=N$, then
\[
 \binom{N}{a,b,c}
 \equiv
 \prod_{j\ge0}\binom{N_j}{a_j,b_j,c_j}
 \pmod 3.
\]
Here a local multinomial coefficient is understood to be zero unless
$a_j+b_j+c_j=N_j$.  In particular, when $N=q-1=3^k-1$, whose $k$ ternary digits are all equal to $2$, the coefficient is nonzero in characteristic three exactly when
\[
 a_j+b_j+c_j=2\qquad(0\le j<k).
\]

\begin{proposition}[Qualitative support gap; Costa--Della Fiore]\label{prop:q3-support-gap}
If $q=3^k$ with $k>1$, then
\[
 H(\Gamma_q)<\log q.
\]
Consequently, with $\rho_q:=\exp(H(\Gamma_q))<q$, Theorem~\ref{thm:support-entropy} gives
\[
 \srk(\tau_q^{\boxt n})
 \le \exp\!\bigl((H(\Gamma_q)+o(1))n\bigr)
 =\rho_q^{\,n+o(n)}.
\]
For $q=3$, one has $H(\Gamma_3)=\log3$.
\end{proposition}

\begin{proof}
The support is invariant under permutations of the three coordinates.  Hence, if a distribution with uniform marginals existed, averaging it over $S_3$ would give an $S_3$-invariant one.  Let $Y=(Y_1,Y_2,Y_3)$ have such a distribution.

There are exactly two support points with first coordinate $1$,
\[
 (1,q-1,q-2),\qquad(1,q-2,q-1).
\]
They lie in the same orbit, so uniformity of the first marginal forces each to have mass $1/(2q)$.

For first coordinate $3$, the possible triples are
\[
 (3,q-1,q-4),\ (3,q-2,q-3),\ (3,q-3,q-2),\ (3,q-4,q-1).
\]
The four corresponding multinomial coefficients are, in the same order,
\[
 \binom{q-1}{3,q-4,0},\quad
 \binom{q-1}{2,q-4,1},\quad
 \binom{q-1}{1,q-4,2},\quad
 \binom{q-1}{0,q-4,3}.
\]
By the multinomial Lucas congruence stated above, the two outer coefficients are nonzero: the ternary addition $3+(q-4)=q-1$ is digitwise and has no carries.  The two middle coefficients vanish, since their least significant ternary digits do not sum to the least significant digit $2$ of $q-1$.  The two surviving points again form one orbit, so each has mass $1/(2q)$.

The first orbit contributes mass $1/q$ to $Y_1=q-1$, and the second contributes another $1/q$, contradicting uniformity.  Proposition~\ref{prop:uniform-marginals} gives $H(\Gamma_q)<\log q$, and Theorem~\ref{thm:support-entropy} gives the slice-rank estimate.

For $q=3$, the support is
\[
 \{(0,0,0),(2,2,0),(0,2,2),(2,0,2),(2,1,1),(1,2,1),(1,1,2)\}.
\]
The probabilities
\[
 \frac14,\ \frac1{12},\ \frac1{12},\ \frac1{12},\ \frac16,\ \frac16,\ \frac16
\]
in this order have uniform marginals.  Hence $H(\Gamma_3)=\log3$.
\end{proof}

\subsection{Ternary support and an optimal dual margin}

\begin{lemma}[Digit description of the characteristic-three support]\label{lem:q3-digit-support}
Let $q=3^k$.  For every integer $a\in\{0,\ldots,q-1\}$, write its unique ternary expansion as
\[
 a=\sum_{j=0}^{k-1}a_j3^j,
 \qquad a_j\in\{0,1,2\},
\]
and define the ternary digit-sum map
\[
 s_3:\{0,\ldots,q-1\}\longrightarrow\{0,\ldots,2k\}\subseteq\mathbb N,
 \qquad
 s_3(a):=\sum_{j=0}^{k-1}a_j.
\]
Here $a$, and later $u,v,w$, are integer exponent indices, not elements of $\F_q$.  Then
\[
 \Gamma_q=\{(0,0,0)\}\cup\Gamma_q^\ast,
\]
where $(u,v,w)\in\Gamma_q^\ast$ if and only if, digit by digit,
\[
 u_j+v_j+w_j=4
 \qquad(0\le j<k).
\]
In particular, every point of $\Gamma_q^\ast$ satisfies
\[
 s_3(u)+s_3(v)+s_3(w)=4k,
\]
and at most one of $u,v,w$ is zero.
\end{lemma}

\begin{proof}
A nonconstant monomial in $\tau_q$ is indexed by integers $a,b,c\in\{0,\ldots,q-1\}$ with
$a+b+c=q-1$ and coefficient
$\binom{q-1}{a,b,c}$.  Applying the multinomial Lucas congruence stated above, and using that every ternary digit of $q-1$ equals $2$, this coefficient is nonzero in characteristic three exactly when
\[
 a_j+b_j+c_j=2
 \qquad(0\le j<k).
\]
Equivalently, the addition $a+b+c=q-1$ is carry-free in base three.  The corresponding exponent triple is
\[
 (u,v,w)=(a+c,a+b,b+c)=(q-1-b,q-1-c,q-1-a).
\]
No borrowing occurs in these subtractions, and hence
\[
 (u_j,v_j,w_j)=(2-b_j,2-c_j,2-a_j),
\]
whose entries sum to $4$.  The converse follows by reversing this construction.  Summing over the digits gives the identity for the digit sums.  Moreover, if two of $u,v,w$ were zero, then the remaining digit would have to equal $4$ in every position, which is impossible; hence at most one coordinate is zero.
\end{proof}

\begin{proposition}[Explicit dual certificate with optimal normalized margin]\label{prop:q3-explicit-dual}
Let $q=3^k$ with $k\ge2$, and set
\[
 d_q:=\frac{q-4}{3q-4}.
\]
Define $w_q:\{0,\ldots,q-1\}\to\mathbb R$ by
\[
 w_q(0)=d_q,
\]
and, for $a\ne0$,
\[
 w_q(a)
 =\frac{3q}{k(3q-4)}s_3(a)
  -\frac{3q+4}{3q-4}.
\]
Then
\[
 \sum_{a=0}^{q-1}w_q(a)=0,
 \qquad
 \lVert w_q\rVert_\infty=1,
\]
and
\[
 w_q(u)+w_q(v)+w_q(w)\ge 3d_q
 \qquad((u,v,w)\in\Gamma_q).
\]
Moreover, the ratio $d_q$ is optimal among all certificates in Proposition~\ref{prop:dual}: if $w_1,w_2,w_3$ and $\eta$ form such a certificate and
$W=\sum_i\lVert w_i\rVert_\infty$, then
\[
 \frac{\eta}{W}\le d_q.
\]
Consequently,
\[
 H(\Gamma_q)
 \le \log q-\frac{d_q^2}{2},
\qquad
 \exp\bigl(H(\Gamma_q)\bigr)
 \le q\exp\!\left(-\frac{d_q^2}{2}\right)<q.
\]
\end{proposition}

\begin{proof}
In the list $0,1,\ldots,3^k-1$, each ternary digit takes the values $0,1,2$ equally often.  Hence
\[
 \sum_{a=0}^{q-1}s_3(a)=qk,
\]
and direct substitution gives $\sum_aw_q(a)=0$.  For $a\ne0$, the value of $w_q(a)$ is increasing with $s_3(a)$.  Its maximum is $1$, attained at $a=q-1$, while its minimum is at least $-1$ because $3q\ge8k$ for $q=3^k$, $k\ge2$.  Together with $0<d_q<1$, this proves $\lVert w_q\rVert_\infty=1$.

At $(0,0,0)$ the sum of the weights is $3d_q$.  If $(u,v,w)\in\Gamma_q^\ast$ has no zero coordinate, Lemma~\ref{lem:q3-digit-support} gives
\[
 w_q(u)+w_q(v)+w_q(w)
 =\frac{3q}{k(3q-4)}\,4k-3\frac{3q+4}{3q-4}=3d_q.
\]
At most one coordinate can be zero; in that case the special value $w_q(0)$ increases the sum by $4q/(3q-4)$.  Thus the certificate inequality holds.

For optimality, start from any certificate $(w_1,w_2,w_3)$ with margin $\eta_0$ and put $W_0=\sum_i\lVert w_i\rVert_\infty$.  Averaging over coordinate permutations gives the symmetric function
\[
 \overline w=\frac{w_1+w_2+w_3}{3},
 \qquad3\lVert\overline w\rVert_\infty\le W_0.
\]
Since $\eta_0>0$, one has $\overline w\not\equiv0$.  After normalization it is enough to consider one zero-mean function $w$ with $\lVert w\rVert_\infty\le1$ and
\[
 w(u)+w(v)+w(w)\ge\eta\quad((u,v,w)\in\Gamma_q),
 \qquad\eta\ge\frac{3\eta_0}{W_0}.
\]
Give mass $2/(3q-4)$ to $(0,0,0)$; with the remaining mass, choose $a$ uniformly from $\{1,\ldots,q-2\}$ and then a random permutation of $(a,q-1-a,q-1)$.  These triples belong to $\Gamma_q$ by Lemma~\ref{lem:q3-digit-support}.  Each marginal $\nu$ satisfies
\[
 \nu(b)=\frac{2}{3q-4}\quad(0\le b\le q-2),
 \qquad \nu(q-1)=\frac{q-2}{3q-4}.
\]
Averaging the certificate inequality and using $\sum_aw(a)=0$ gives
\[
 \eta\le3\sum_a\nu(a)w(a)
 =3\frac{q-4}{3q-4}w(q-1)\le3d_q.
\]
Therefore $\eta_0/W_0\le d_q$, and the displayed certificate attains equality.  Proposition~\ref{prop:dual} gives the entropy estimate.
\end{proof}

\subsection{The sharpened analytic base}

The preceding Pinsker estimate is completely explicit, but the same certificate gives a substantially sharper one-dimensional analytic bound.

\begin{proposition}[Sharpened analytic base]\label{prop:q3-sharp-base}
For $k\ge2$, define
\[
 C_k:=\inf_{x\ge1}
 \left\{
 1+x^{-4k/3}\bigl((1+x+x^2)^k-1\bigr)
 \right\}.
\]
Then
\[
 H(\Gamma_{3^k})\le\log C_k,
 \qquad C_k<3^k.
\]
In particular, with
\[
 x_0:=\frac{1+\sqrt{33}}4,
 \qquad
 \gamma:=(1+x_0+x_0^2)x_0^{-4/3}
          \approx2.755104613024,
\]
one has the fully explicit estimate
\[
 C_k
 \le1+\gamma^k-x_0^{-4k/3}
 <1+\gamma^k.
\]
\end{proposition}

\begin{proof}
Let $\mu$ be a probability distribution on $\Gamma_q$.  Since
\[
 \sum_{i=1}^3w_q(\alpha_i)\ge3d_q
\]
on the support, at least one marginal $\nu$ satisfies
$\mathbb E_\nu w_q\ge d_q$.  For every $\lambda\ge0$, the standard entropy variational inequality~\cite{CoverThomas} gives
\[
 H(\nu)
 \le\log\left(\sum_{a=0}^{q-1}e^{\lambda w_q(a)}\right)
      -\lambda d_q.
\]
Using the explicit formula for $w_q$ and setting
\[
 x=\exp\!\left(\frac{3q\lambda}{k(3q-4)}\right)\ge1,
\]
the right-hand side becomes the logarithm of
\[
 1+x^{-4k/3}\bigl((1+x+x^2)^k-1\bigr).
\]
As $\lambda$ ranges over $[0,\infty)$, this change of variables ranges bijectively over $x\in[1,\infty)$.  Taking the infimum proves the first assertion.  At $x=1$ this expression equals $3^k$, while its derivative there is
\[
 \frac{k(4-3^k)}3<0,
\]
so $C_k<3^k$.

Finally, $x_0$ is the positive minimizer on $[1,\infty)$ of
$(1+x+x^2)x^{-4/3}$.  Substitution gives
\[
 C_k
 \le1+x_0^{-4k/3}\bigl((1+x_0+x_0^2)^k-1\bigr)
 =1+\gamma^k-x_0^{-4k/3}.
\]
\end{proof}

\begin{remark}[Numerical values]\label{rem:q3-dual-certificate}
Since the minimization is one-dimensional, $C_k$ is one-dimensional can be computed numerically also for large $k$.  Some values, including cases beyond $k=5$, are as follows.  The column labelled ``dual--Pinsker'' is the explicit base
$3^k\exp(-d_{3^k}^2/2)$ obtained by combining the certificate with optimal normalized margin from Proposition~\ref{prop:q3-explicit-dual} with Pinsker's inequality; the sharper column is $C_k$.
\[
\begin{array}{c|r|c|r|r}
 k & q=3^k & 3d_q & \text{dual--Pinsker} & C_k\\
\hline
 2 & 9    & 0.652173913 & 8.789827507 & 8.311438619\\
 3 & 27   & 0.896103896 & 25.821968790 & 21.785657318\\
 4 & 81   & 0.966527197 & 76.903435734 & 58.555178033\\
 5 & 243  & 0.988965517 & 230.148597638 & 159.710734223\\
 6 & 729  & 0.996335318 & 689.884793146 & 438.334118263\\
 7 & 2187 & 0.998779930 & 2069.093661669 & 1205.935801600\\
 8 & 6561 & 0.999593475 & 6206.720366102 & 3320.741398379\\
 9  & 19683    & 0.999864510 & 18619.600512887 & 9147.243417100 \\
10 & 59049    & 0.999954839 & 55858.240964466 & 25199.861703687 \\
11 & 177147   & 0.999984947 & 167574.162322948 & 69426.502252124 \\
12 & 531441   & 0.999994982 & 502721.926399646 & 191275.522577324 \\
13 & 1594323  & 0.999998327 & 1508165.218630154 & 526982.320035362 \\
14 & 4782969  & 0.999999443 & 4524495.095321819 & 1451889.666070144 \\
15 & 14348907 & 0.999999814 & 13573484.725396864 & 4000106.161617847 \\

\end{array}
\]
The normalized LP for the dual certificate has optimal margin $3d_q$; the displayed analytic certificate therefore removes the need to solve a new linear program for each $k$.  For $2\le k\le5$, the values of $C_k$ agree numerically, to the precision reported in~\cite{CostaDellaFiore}, with the corresponding field-by-field support-entropy optimizations contained in the published article.  That article reported the entropy bounds
\[
 2.118,\qquad 3.082,\qquad 4.070,\qquad 5.074
\]
for $k=2,3,4,5$, respectively, and hence the exponential bases
\[
 8.315,\qquad 21.802,\qquad 58.557,\qquad 159.812.
\]
By comparison,
\[
 \log C_2\approx2.117632713,\quad
 \log C_3\approx3.081251832,\quad
 \log C_4\approx4.069969524,\quad
 \log C_5\approx5.073364268.
\]
The entropy bounds reported in~\cite{CostaDellaFiore} are obtained by rounding these values upward to three decimal places.  Thus the smaller decimals in the present table should not be interpreted as evidence of a different support-entropy optimum in these four cases.  The gain is instead a uniform one-dimensional analytic formula valid for every $k\ge2$, avoiding field-by-field convex optimization.  In particular, while the field-by-field computation in~\cite{CostaDellaFiore} was carried out only for $2\le k\le5$ and was reported to become impractical beyond $q=243$, the table above also provides explicit values for $6\le k\le15$.  Numerically, the optimizing $x$ converges rapidly to $x_0$.
\end{remark}

\begin{lemma}[Zero-weight diagonal extraction]\label{lem:q3-diagonal-extraction}
Let $q=3^k$, and let $S$ be a sequence in $\F_q^n$ containing no three distinct indices on which $e_2$ vanishes coordinatewise.  Then there is a set $A_0\subseteq\F_q^n$ with
\[
 |A_0|\ge\frac{|S|}{2(n+1)}
\]
such that $\tau_q^{\boxt n}$ restricts to a diagonal tensor with nonzero diagonal on $A_0^3$.
\end{lemma}

\begin{proof}
Every value occurs at most twice in $S$, since $e_2(x,x,x)=3x^2=0$ in characteristic three.  {Let $A$ be the set of distinct values appearing in $S$.  Then $|A|\ge|S|/2$.}  Partition $A$ according to the number of zero coordinates and choose a class $A_0$ of size at least $|A|/(n+1)$.

Tuples of three pairwise distinct elements of $A_0$ do not annihilate $e_2$ by hypothesis.  Suppose next that $x,z\in A_0$ are distinct.  In characteristic three,
\[
 e_2(x_i,x_i,z_i)=x_i^2+2x_iz_i=x_i(x_i-z_i).
\]
If $e_2(x,x,z)$ vanished coordinatewise, then $z_i=x_i$ wherever $x_i\ne0$.  Since $x$ and $z$ have the same number of zero coordinates, this would force $z=x$, a contradiction.  The same argument applies to the other two partial diagonals.  Finally, constant triples lie in the zero set of $e_2$, so $\tau_q^{\boxt n}$ has nonzero diagonal on $A_0^3$.
\end{proof}

Lemma~\ref{lem:q3-diagonal-extraction}, Corollary~\ref{cor:subexp-extraction} with $M_n=2(n+1)$, and Proposition~\ref{prop:q3-sharp-base} yield the following explicit refinement of the result from the arXiv version of~\cite{CostaDellaFiore}.  Proposition~\ref{prop:q3-explicit-dual} gives a closed-form separating certificate with optimal normalized margin, so no field-by-field linear program is needed.

\begin{corollary}\label{cor:original-310}
Let $q=3^k$ with $k>1$, and let $C_k$ be as in Proposition~\ref{prop:q3-sharp-base}.  Then
\[
 \EGZ(3,\F_q^n,2)\le C_k^{\,n+o(n)},
 \qquad C_k<q.
\]
In particular,
\[
 \EGZ(3,\F_{3^k}^n,2)
 \le\bigl(1+\gamma^k-x_0^{-4k/3}\bigr)^{n+o(n)},
\]
with $x_0$ and $\gamma$ as in Proposition~\ref{prop:q3-sharp-base}.
\end{corollary}

{
Corollary~\ref{cor:original-310} gives a strictly smaller exponential base than $1+\gamma^k$, and hence improves the bound from~\cite[Theorem~4.3]{CostaDellaFiore} $$(1+\gamma^k)^{n+o(n)}.$$}

\section{Partition-rank criteria for distinct-variable avoidance}\label{sec:partition-criteria}

Strong freeness makes the tensor detecting coordinatewise polynomial zeros diagonal and permits a direct slice-rank argument.  Pairwise distinct-variable avoidance is subtler because every nontrivial equality profile creates a partial diagonal.  The partition-rank diagonalization used in this section is due to Naslund~\cite{NaslundPartition}, within the partition-lattice framework developed by Omar~\cite{OmarPartition}.  In the formulation developed here, each equality profile is associated with the tensor detecting the coordinatewise zeros of the corresponding polynomial contraction.  Lemma~\ref{lem:lifting} and Theorem~\ref{thm:distinct-reduction} reduce the global distinct-variable problem to these local tensors, {which are then controlled by degree bounds or eliminated when the corresponding contractions have only diagonal zeros, particularly on restricted alphabets.}
\subsection{Distinctness indicators and contracted tensors}
{ Naslund introduced partition rank and proved its diagonal lemma~\cite[Lemma~11]{NaslundPartition}.  His modified distinctness indicator is given in~\cite[Lemmas~14--15]{NaslundPartition}; grouping permutations by their cycle partitions yields the equality-profile expansion in~\cite[Eq.~(3.2)]{NaslundPartition}.}  The same coefficients are the M\"obius coefficients of the partition lattice, as reviewed and generalized by Omar~\cite[Section~2.1 and Theorem~8]{OmarPartition}.  We use Naslund's normalization because its coefficients interact explicitly with polynomial contractions.  Over a field $\mathbb K$, a tensor $T:X_1\times\cdots\times X_r\to\mathbb K$ has \emph{partition rank one} if, for some nonempty proper subset $I\subset[r]$, it factors as
\[
 T(x_1,\ldots,x_r)=
 f((x_i)_{i\in I})g((x_i)_{i\notin I}).
\]
Its \emph{partition rank} $\prk(T)$ is the least number of partition-rank-one tensors whose sum is $T$.  Every slice has partition rank one, so
\[
 \prk(T)\le\srk(T).
\]
For tensors of order two, both notions are ordinary matrix rank.  We shall use the partition-rank version of the diagonal-tensor lemma, due to Naslund~\cite[Lemma~11]{NaslundPartition}: if $D$ is diagonal with $|A|$ nonzero diagonal entries, then
\[
 \prk(D)=|A|.
\]

Let $\Pi_t$ be the lattice of set partitions of $[t]$, ordered by refinement, and write $|\pi|$ for the number of blocks of $\pi\in\Pi_t$.  Its least and greatest elements are
\[
 \widehat0=\bigl\{\{1\},\ldots,\{t\}\bigr\},
 \qquad
 \widehat1=\bigl\{[t]\bigr\}.
\]
Let $X$ be any set and let $\mathbb K$ be a field.  For $\pi\in\Pi_t$, define the equality indicator
\[
 \Delta_\pi:X^t\longrightarrow\mathbb K,
 \qquad
 \Delta_\pi(x_1,\ldots,x_t)=
 \prod_{B\in\pi}\1_{\{x_i:i\in B\}\text{ are all equal}},
\]
where $\1_E\in\{0,1\}\subseteq\mathbb K$ denotes the indicator of the condition $E$.  If $G:X^t\to\mathbb K$ is a tensor, then $\Delta_\pi G$ always denotes the pointwise product
\[
 (\Delta_\pi G)(\mathbf x)=\Delta_\pi(\mathbf x)G(\mathbf x).
\]
The \emph{equality partition} of $\mathbf x=(x_1,\ldots,x_t)$ is the partition of $[t]$ whose blocks are the fibers of the map $i\mapsto x_i$.  Thus $\Delta_\pi(\mathbf x)=1$ if and only if $\pi$ refines the equality partition of $\mathbf x$.  In the applications below we take $\mathbb K=\F_q$ and $X=\F_q^n$, so $\Delta_\pi$ and the coordinatewise polynomial tensors have the same domain $(\F_q^n)^t$.

We use Naslund's distinctness indicator in the following normalization~\cite[Lemma~15]{NaslundPartition}.  Let $S_t$ denote the symmetric group on $[t]$.  If $\sigma\in S_t$, let $f_\sigma:X^t\to\mathbb K$ be the indicator that $\mathbf x$ is fixed by $\sigma$, and let $\mathrm{Cyc}_t$ be the set of $t$-cycles.  Define $H_t:X^t\to\mathbb K$ by
\[
 H_t(\mathbf x)=
 \sum_{\sigma\in S_t\setminus\mathrm{Cyc}_t}
 \operatorname{sgn}(\sigma)f_\sigma(\mathbf x).
\]

\begin{lemma}[Naslund's distinctness indicator]\label{lem:naslund-indicator}
For every $\mathbf x=(x_1,\ldots,x_t)$,
\[
 H_t(\mathbf x)=
 \begin{cases}
 1,&x_1,\ldots,x_t\text{ are pairwise distinct},\\
 (-1)^t(t-1)!,&x_1=\cdots=x_t,\\
 0,&\text{otherwise}.
 \end{cases}
\]
Moreover,
\[
 H_t=\sum_{\pi\in\Pi_t\setminus\{\widehat1\}}c_\pi\Delta_\pi,
\qquad
 c_\pi=(-1)^{t-|\pi|}\prod_{B\in\pi}(|B|-1)!.
\]
\end{lemma}

\begin{proof}
{Indeed,
\[
 \sum_{\sigma\in S_t}\operatorname{sgn}(\sigma)f_\sigma(x)
 =
 \sum_{\sigma\in\operatorname{Stab}(x)}\operatorname{sgn}(\sigma).
\]
If the entries of \(x\) are pairwise distinct, then the stabilizer is
trivial and this sum equals \(1\). Otherwise, a transposition exchanging
two equal entries belongs to the stabilizer, and multiplication by this
transposition pairs its elements with opposite signs. Hence the sum is
zero. }

Removing the $t$-cycles changes the value only when the equality partition is $\widehat1$, because a $t$-cycle can fix $\mathbf x$ only when all entries are equal.  There are $(t-1)!$ such cycles, all of sign $(-1)^{t-1}$.  Since the alternating sum over the whole symmetric group is zero on a constant tuple, the remaining sum is
\[
 -(-1)^{t-1}(t-1)!=(-1)^t(t-1)!.
\]
This proves the values of $H_t$.

For a permutation whose cycles are the blocks of $\pi$, one has $f_\sigma=\Delta_\pi$.  On a block $B$ of size $b$, there are $(b-1)!$ cycles and each has sign $(-1)^{b-1}$.  Multiplying over blocks gives the displayed coefficient.  The one-block partition is absent because the $t$-cycles were removed.
\end{proof}

Let $F\in\F_q[x_1,\ldots,x_t]$ satisfy $F(a,\ldots,a)=0$ for every $a\in\F_q$.  Its scalar zero indicator and its coordinatewise tensor power are
\[
 \tau_F:\F_q^t\longrightarrow\F_q,
 \qquad
 \tau_F=1-F^{q-1},
\]
\[
 \Tau_{F,n}:=\tau_F^{\boxt n}:(\F_q^n)^t\longrightarrow\F_q.
\]
Thus $\Tau_{F,n}(x^{(1)},\ldots,x^{(t)})=1$ precisely when
$F(x^{(1)},\ldots,x^{(t)})=0$ in every coordinate.

For a partition $\pi=\{B_1,\ldots,B_r\}$, choose an ordering of its blocks and define the contraction
\[
 F_\pi(y_1,\ldots,y_r)
 :=F(z_1,\ldots,z_t),
 \qquad z_j=y_i\text{ if }j\in B_i.
\]
Set
\[
 \tau_{F,\pi}:=1-F_\pi^{q-1}:\F_q^r\longrightarrow\F_q,
 \qquad
 \Tau_{F,\pi,n}:=\tau_{F,\pi}^{\boxt n}:(\F_q^n)^r\longrightarrow\F_q.
\]
Reordering the blocks only permutes the variables of these tensors and therefore does not affect their slice or partition rank.

{We now relate each equality-profile term in Naslund's expansion to the
tensor associated with the corresponding polynomial contraction.}
\begin{lemma}[Lifting contractions]\label{lem:lifting}
If $\pi$ has at least two blocks, then
\[
 \prk(\Delta_\pi\Tau_{F,n})
 \le\prk(\Tau_{F,\pi,n})
 \le\srk(\Tau_{F,\pi,n}).
\]
\end{lemma}

\begin{proof}
Write $\pi=\{B_1,\ldots,B_r\}$ with $r\ge2$, and choose a representative $b_i\in B_i$ for each block.  { On every tuple for which \(\Delta_\pi=1\), all variables in \(B_i\)
equal \(x^{(b_i)}\), and hence} the definitions of contraction and coordinatewise tensor power give
\[
 \Tau_{F,n}(x^{(1)},\ldots,x^{(t)})
 =\Tau_{F,\pi,n}(x^{(b_1)},\ldots,x^{(b_r)}).
\]
Equivalently, for all tuples of vector variables,
\[
 \Delta_\pi\Tau_{F,n}
 =\Delta_\pi\,
 \Tau_{F,\pi,n}(x^{(b_1)},\ldots,x^{(b_r)}),
\]
where the second factor on the right is the pullback of the contracted tensor through the representative coordinates.

Suppose
\[
 \Tau_{F,\pi,n}(y_1,\ldots,y_r)
 =\sum_{s=1}^R
 f_s((y_i)_{i\in I_s})\,
 g_s((y_i)_{i\notin I_s})
\]
is a partition-rank decomposition, with $\varnothing\ne I_s\subsetneq[r]$.  Let
$U_s=\bigcup_{i\in I_s}B_i$, a nonempty proper subset of $[t]$.  Define
\[
 \widetilde f_s((x^{(j)})_{j\in U_s})
 :=f_s((x^{(b_i)})_{i\in I_s})
 \prod_{i\in I_s}\prod_{j\in B_i}
 \1_{x^{(j)}=x^{(b_i)}},
\]
\[
 \widetilde g_s((x^{(j)})_{j\notin U_s})
 :=g_s((x^{(b_i)})_{i\notin I_s})
 \prod_{i\notin I_s}\prod_{j\in B_i}
 \1_{x^{(j)}=x^{(b_i)}}.
\]
{ Since \(I_s\) and \([r]\setminus I_s\) partition the set of blocks,
the product of the two equality factors is
\[
 \prod_{i=1}^r\prod_{j\in B_i}
 \mathbf 1_{x^{(j)}=x^{(b_i)}}
 =\Delta_\pi.
\]}  Therefore
\[
 \Delta_\pi\Tau_{F,n}
 =\sum_{s=1}^R\widetilde f_s\widetilde g_s.
\]
Each summand has partition rank one with respect to the bipartition
$U_s\mid([t]\setminus U_s)$.  Thus
$\prk(\Delta_\pi\Tau_{F,n})\le R$, and minimizing over decompositions proves the first inequality.  The second follows because every slice is a partition-rank-one tensor.
\end{proof}

The diagonalization step in the next theorem is Naslund's distinctness-indicator argument and is also a specialization of Omar's general partition-indicator theorem~\cite[Theorem~8]{OmarPartition}.  The contracted-tensor right-hand side is the profile-wise reformulation used in this paper.

\begin{theorem}[Distinct-variable reduction via contractions]\label{thm:distinct-reduction}
Let $q$ have characteristic $p$, assume $2\le t\le p$, and let $F\in\F_q[x_1,\ldots,x_t]$ satisfy
\[
 F(a,\ldots,a)=0\qquad(a\in\F_q).
\]
If $A\subseteq\F_q^n$ contains no $t$ pairwise distinct elements satisfying $F=0$ coordinatewise, then
\[
 |A|
 \le
 \sum_{\pi\in\Pi_t\setminus\{\widehat1\}}
 \srk(\Tau_{F,\pi,n}).
\]
Consequently, if a sequence $S$ contains no length-$t$ subsequence on which $F$ vanishes, then
\[
 |S|
 \le
 (t-1)
 \sum_{\pi\in\Pi_t\setminus\{\widehat1\}}
 \srk(\Tau_{F,\pi,n}).
\]
\end{theorem}

\begin{proof}
Consider
\[
 J=H_t\Tau_{F,n}.
\]
On $A^t$, a tuple of pairwise distinct elements has
$\Tau_{F,n}=0$ by hypothesis; a partial diagonal has $H_t=0$;
and a constant tuple has
\[
 J(x,\ldots,x)=(-1)^t(t-1)!.
\]
Since $t\le p$, this scalar is nonzero in $\F_q$. Hence
$J|_{A^t}$ is diagonal with nonzero diagonal. By Naslund's
partition-rank diagonal lemma,
\[
 |A|\le\prk(J).
\]

By Lemma~\ref{lem:naslund-indicator}, with the integer coefficients
$c_\pi$ viewed in $\F_q$,
\[
 J
 =H_t\Tau_{F,n}
 =\sum_{\pi\in\Pi_t\setminus\{\widehat1\}}
 c_\pi\Delta_\pi\Tau_{F,n}.
\]
{Subadditivity of partition rank gives
\[
 \prk(J)
 \le
 \sum_{\pi\in\Pi_t\setminus\{\widehat1\}}
 \prk(c_\pi\Delta_\pi\Tau_{F,n})
 \le
 \sum_{\pi\in\Pi_t\setminus\{\widehat1\}}
 \prk(\Delta_\pi\Tau_{F,n}),
\]
because multiplication by a nonzero scalar does not change partition
rank, while if $c_\pi=0$ in $\F_q$ the corresponding term on the
left is zero. Applying Lemma~\ref{lem:lifting}, we obtain
\[
 |A|
 \le\prk(J)
 \le\sum_{\pi\ne\widehat1}\prk(\Delta_\pi\Tau_{F,n})
 \le\sum_{\pi\ne\widehat1}\srk(\Tau_{F,\pi,n}).
\]
Let \(A\) be the set of distinct values appearing in \(S\).
Since no value can occur \(t\) times, we have
\[
 |A|\ge \frac{|S|}{t-1}.
\]
If \(A\) contained \(t\) distinct elements satisfying \(F=0\),
choosing one occurrence of each in \(S\) would give a forbidden
subsequence. Apply the set bound to \(A\).}
\end{proof}

\subsection{\texorpdfstring{Degree cutoffs and contractions with diagonal zero sets}{Degree cutoffs and contractions with diagonal zero sets}}
The degree method automatically controls contractions having many blocks.  The following cutoff is an immediate application of the standard low-degree slice bound, not a new rank estimate.

\begin{corollary}[Degree cutoff]\label{cor:degree-cutoff}
Under the hypotheses of Theorem~\ref{thm:distinct-reduction}, suppose that $F$ has total degree $d$.  If $\pi$ has $r$ blocks, then
\[
 \srk(\Tau_{F,\pi,n})
 \le
 rM_q\left(n,\left\lfloor\frac{d(q-1)n}{r}\right\rfloor\right).
\]
If $r>2d$, this is at most $c_{q,d,r}^{\,n+o(n)}$ for some $c_{q,d,r}<q$.
\end{corollary}

\begin{proof}
The contraction $F_\pi$ has degree at most $d$, so $\Tau_{F,\pi,n}$ has a reduced representative of total degree at most $d(q-1)n$.  Apply Lemma~\ref{lem:degree-slice}.  The strict exponential improvement follows because
\[
 \frac{d(q-1)}r<\frac{q-1}{2}.
\]
\end{proof}

Thus, for a degree-$d$ polynomial, only equality profiles involving at most $2d$ distinct values can obstruct a nontrivial exponential bound.  {The next refinement removes those equality profiles for which the corresponding contracted polynomial has no non-diagonal zeros. The proposition below applies Omar's partition-lattice argument~\cite[Theorem~8]{OmarPartition} to polynomial contractions, using Naslund's coefficients.}

\begin{proposition}[{Eliminating profiles with diagonal zero sets}]\label{prop:diagonal-zero-refinement}
Let $X\subseteq\F_q$, let $t\ge2$, and let $F\in\F_q[x_1,\ldots,x_t]$ satisfy
\[
 F(a,\ldots,a)=0\qquad(a\in X).
\]
Let $\mathcal R\subseteq\Pi_t\setminus\{\widehat1\}$ be a collection of partitions such that, for every $\pi=\{B_1,\ldots,B_r\}\in\mathcal R$,
\[
 F_\pi(y_1,\ldots,y_r)=0
 \quad\Longrightarrow\quad
 y_1=\cdots=y_r
 \qquad(y_i\in X).
\]
Let $c_\pi$ be the coefficients from Lemma~\ref{lem:naslund-indicator}, and assume that
\[
 \beta:=(-1)^t(t-1)!-\sum_{\pi\in\mathcal R}c_\pi
\]
is nonzero in $\F_q$.  If $A\subseteq X^n$ contains no coordinatewise zero of $F$ formed by $t$ pairwise distinct elements, then
\[
 |A|\le
 \sum_{\pi\in\Pi_t\setminus(\mathcal R\cup\{\widehat1\})}
 \srk\!\left(\left.\Tau_{F,\pi,n}\right|_{(X^n)^{|\pi|}}\right).
\]
The analogous bound for sequences over $X^n$ holds after multiplication by $t-1$.
\end{proposition}

\begin{proof}
Fix $\pi=\{B_1,\ldots,B_r\}\in\mathcal R$.  We first justify the identity
\[
 \left.\Delta_\pi\Tau_{F,n}\right|_{(X^n)^t}
 =\left.\Delta_{\widehat1}\right|_{(X^n)^t}.
\]
If the left-hand side is nonzero at a tuple
$(x^{(1)},\ldots,x^{(t)})\in(X^n)^t$, then $\Delta_\pi=1$, so there are block values
$y_1,\ldots,y_r\in X^n$ such that $x^{(j)}=y_i$ for $j\in B_i$.  Moreover, $\Tau_{F,n}=1$, so for every coordinate $h\in[n]$,
\[
 F_\pi(y_{1,h},\ldots,y_{r,h})=0.
\]
{The hypothesis on \(F_\pi\) implies
\(y_{1,h}=\cdots=y_{r,h}\) for every \(h\), hence
\(y_1=\cdots=y_r\) as vectors and the original tuple is constant.}  Conversely, a constant tuple in $(X^n)^t$ satisfies $\Delta_\pi=1$ and $\Tau_{F,n}=1$ because $F(a,\ldots,a)=0$ for every $a\in X$.  This proves the identity.

Using the equality-profile expansion of Lemma~\ref{lem:naslund-indicator}, we have on $(X^n)^t$
\[
 \begin{aligned}
 J=H_t\Tau_{F,n}
 &=\sum_{\pi\in\Pi_t\setminus\{\widehat1\}}
 c_\pi\Delta_\pi\Tau_{F,n}\\
 &=\left(\sum_{\pi\in\mathcal R}c_\pi\right)\Delta_{\widehat1}
 +\sum_{\pi\in\Pi_t\setminus(\mathcal R\cup\{\widehat1\})}
 c_\pi\Delta_\pi\Tau_{F,n}.
 \end{aligned}
\]
On $A^t$, the same diagonalization argument used in Theorem~\ref{thm:distinct-reduction} gives
\[
 J\big|_{A^t}=(-1)^t(t-1)!\,\Delta_{\widehat1}\big|_{A^t}.
\]
Subtracting the first term in the preceding expansion from both sides yields
\[
 \begin{aligned}
 \beta\,\Delta_{\widehat1}\big|_{A^t}
 &=\left((-1)^t(t-1)!-\sum_{\pi\in\mathcal R}c_\pi\right)
   \Delta_{\widehat1}\big|_{A^t}\\
 &=\sum_{\pi\in\Pi_t\setminus(\mathcal R\cup\{\widehat1\})}
 c_\pi\bigl(\Delta_\pi\Tau_{F,n}\bigr)\big|_{A^t}.
 \end{aligned}
\]
{Thus the coefficient produced by moving the terms indexed by $\mathcal R$ to the left-hand side is exactly $\beta$.}  Since $\beta\ne0$, the tensor on the left is a nonzero scalar multiple of the diagonal tensor on $A^t$ and therefore has partition rank $|A|$.  Applying subadditivity, discarding coefficients that vanish in $\F_q$, and then applying the lifting argument of Lemma~\ref{lem:lifting} after restriction to $X^n$ proves the set bound.  The sequence statement follows exactly as in Theorem~\ref{thm:distinct-reduction}.
\end{proof}

\section{Restricted-alphabet zero-sum problems}\label{sec:zero-sum}

We now apply the partition-rank criteria to the linear polynomial $e_1$ on restricted alphabets.  {On the full alphabet, the same profile-wise argument recovers, up to a $p$-dependent prefactor, the exponential bound in Naslund's classical $p$-term zero-sum result~\cite{NaslundEGZ}.}  For the sequence problem, Fox and Sauermann obtained a sharper bound by relating the Erd\H{o}s--Ginzburg--Ziv constant to three-term-progression-free sets~\cite{FoxSauermann}.  Since the full-alphabet problem is already well studied, we pass directly to the multiplicative torus, where retaining the alphabet $\F_q^\times$ produces genuinely different estimates.

\subsection{The multiplicative torus}

Let $X\subseteq\F_q$.  We call a set $A\subseteq X^n$ \emph{strongly $F$-free relative to $X$} if every coordinatewise zero of $F$ in $A^t$ is constant, and \emph{distinctly $F$-free relative to $X$} if it contains no coordinatewise zero formed by $t$ pairwise distinct elements.  The corresponding sequence problem is a restricted-alphabet EGZ-type problem.  When $X=\F_q^\times$, so that
\[
 X^n=(\F_q^\times)^n=\mathbb G_m^n(\F_q)
\]
is the set of $\F_q$-rational points of the split algebraic torus, we use the adjective \emph{toric} for the corresponding restricted-alphabet problems.

The results below give nontrivial bounds for the linear strong and distinct-variable problems on the torus by applying the criteria of Section~\ref{sec:partition-criteria} to the restricted alphabet.  Coordinatewise inversion then transfers them to $e_{p-1}$, and support stratification lifts the toric estimates to the whole vector space.  The final consequence is a new nontrivial bound for $\operatorname{EGZ}(5,\F_5^n,4)$ and a full-space strong $e_4$ bound outside the range of Corollary~\ref{thm:strong-em}.

Recall the notation $M_Q(n,D)$ introduced in Section~\ref{sec:slice-criteria}.  Let $X=\F_q^\times$, so $|X|=q-1$.  Every function on $X$ has a unique representative in the basis
\[
 1,x,\ldots,x^{q-2},
\]
because $x^{q-1}=1$ on $X$.

The following lemma is the immediate restricted-alphabet version of Lemma~\ref{lem:degree-slice}; the proof is the same standard degree splitting with the monomial basis of functions on $\F_q^\times$.

\begin{lemma}[Low-degree slice bound on the multiplicative torus]\label{lem:torus-degree}
Let $P$ be a polynomial function on $(X^n)^r$.  Suppose that, after reducing every scalar exponent modulo $q-1$, it has a representative of total degree at most $D$.  Then
\[
 \srk(P)\le rM_{q-1}\left(n,\left\lfloor\frac Dr\right\rfloor\right).
\]
\end{lemma}

\begin{proof}
Expand $P$ in the product basis consisting of monomials with exponents in $\{0,\ldots,q-2\}$.  In every monomial, at least one of the $r$ vector-variable groups has degree at most $D/r$.  Assign the monomial to the first such group and collect terms with the same monomial factor in that group.  This is exactly the proof of Lemma~\ref{lem:degree-slice}, with the alphabet size $q-1$ in place of $q$.
\end{proof}

The same standard torus degree estimate gives the following immediate restricted-alphabet strong-freeness consequence for the linear polynomial $e_1$.  It will later be transported to the higher-degree polynomial $e_{p-1}$.

Let $q=p^k>4$, where $p$ is an odd prime, and define
\[
 \eta_{q,p}:=
 \inf_{0<\rho<1}
 (1+\rho+\cdots+\rho^{q-2})
 \rho^{-(q-1)/p}.
\]
Since
\[
 \frac{q-1}{p}<\frac{q-2}{2},
\]
the usual derivative argument gives $\eta_{q,p}<q-1$.

\begin{corollary}[Toric strong $e_1$-freeness]\label{cor:toric-strong-e1}
Let $q=p^k>4$, with $p$ odd.  If $A\subseteq(\F_q^\times)^n$ is strongly $e_1$-free of length $p$ relative to the torus, then
\[
 |A|\le
 pM_{q-1}\left(n,\left\lfloor\frac{(q-1)n}{p}\right\rfloor\right)
 \le p\eta_{q,p}^n,
 \qquad \eta_{q,p}<q-1.
\]
Thus the strong $e_1$-free problem has a nontrivial exponential bound measured against the torus alphabet size.
\end{corollary}

\begin{proof}
On the torus, the local zero indicator
\[
 1-(x_1+\cdots+x_p)^{q-1}
\]
has a reduced representative of total degree at most $q-1$.  Since $p=0$ in $\F_q$, the diagonal is contained in the zero set of $e_1$.  On a strongly $e_1$-free set, the coordinatewise tensor power of this indicator is therefore diagonal with nonzero diagonal.  Apply the diagonal-tensor lemma and Lemma~\ref{lem:torus-degree} with $r=p$ and $D=(q-1)n$ to obtain the first inequality.

For the second inequality, let
\[
 G_{q-1}(\rho):=1+\rho+\cdots+\rho^{q-2}.
\]
The power $G_{q-1}(\rho)^n$ is the generating function for exponent vectors in $\{0,\ldots,q-2\}^n$, with the exponent of $\rho$ recording their total degree.  For $0<\rho<1$ and
$D=\lfloor(q-1)n/p\rfloor$, every vector counted by $M_{q-1}(n,D)$ contributes at least $\rho^D$.  Hence
\[
 M_{q-1}(n,D)\rho^D
 \le G_{q-1}(\rho)^n,
\]
and, since $D\le(q-1)n/p$,
\[
 M_{q-1}(n,D)
 \le\left(G_{q-1}(\rho)\rho^{-(q-1)/p}\right)^n.
\]
Taking the infimum over $0<\rho<1$ gives
$M_{q-1}(n,D)\le\eta_{q,p}^n$.
\end{proof}

We next consider the distinct-variable, or toric EGZ-type, version.  Define
\[
 \vartheta_q:=
 \inf_{0<\rho<1}
 (1+\rho+\cdots+\rho^{q-2})
 \rho^{-(q-1)/3}.
\]
Since
\[
 \frac{q-1}{3}<\frac{q-2}{2}
\]
for $q>4$, the same derivative argument used for $\kappa_q(\alpha)$ gives
\[
 \vartheta_q<q-1.
\]
Put
\[
 B_p=\sum_{r=3}^p r\StirlingTwo{p}{r}.
\]

The next theorem is the restricted-alphabet analogue of Naslund's partition-rank proof for $p$-term zero sums~\cite{NaslundEGZ}.  The new point is to retain the torus alphabet throughout, so that the exponential base is compared with $q-1$ rather than with $q$.

\begin{theorem}[Toric distinct zero-sum bound]\label{thm:toric-zero-sum}
Let $q=p^k>4$, with $p$ odd, and let $A\subseteq(\F_q^\times)^n$.  If $A$ is distinctly $e_1$-free of length $p$ relative to the torus, equivalently if it contains no $p$ pairwise distinct elements whose sum is zero, then
\[
 |A|\le
 \sum_{r=3}^p r\StirlingTwo{p}{r}
 M_{q-1}\left(n,\left\lfloor\frac{(q-1)n}{r}\right\rfloor\right)
 \le B_p\vartheta_q^n.
\]
Consequently, if a sequence over $(\F_q^\times)^n$ contains no zero-sum subsequence of length $p$, then its length is at most
\[
 (p-1)B_p\vartheta_q^n.
\]
In particular, the exponential base is strictly smaller than the torus alphabet size $q-1$.
\end{theorem}

\begin{proof}
Apply Proposition~\ref{prop:diagonal-zero-refinement} with
\[
 X=\F_q^\times,\qquad t=p,\qquad
 F(x_1,\ldots,x_p)=x_1+\cdots+x_p.
\]  {Every two-block contraction has only diagonal zeros:} if the block sizes are $a$ and $p-a$, then
\[
 F_\pi(y_1,y_2)=a(y_1-y_2),
\]
and $a\ne0$ in characteristic $p$.  Let $\mathcal R$ be the family of two-block partitions.  The sum of the coefficients $c_\pi$ over $\mathcal R$ is the signed sum of permutations of $[p]$ having exactly two cycles.  Hence it equals
\[
 (-1)^{p-2}\StirlingOne{p}{2}
 =-(p-1)!\sum_{j=1}^{p-1}\frac1j=0
\]
in $\F_p$, where the last equality follows by pairing $j$ with $p-j$.  Wilson's theorem gives $(-1)^p(p-1)!=1$, so the coefficient $\beta$ in Proposition~\ref{prop:diagonal-zero-refinement} is one.

Let $\pi$ have $r\ge3$ blocks of sizes $\lambda_1,\ldots,\lambda_r$.  Its contraction is the nonzero linear form
\[
 L_\pi(y_1,\ldots,y_r)=\lambda_1y_1+\cdots+\lambda_ry_r.
\]
On $X^r$, its zero indicator is
\[
 \tau_\pi=1-L_\pi^{q-1}.
\]
Using $y_i^{q-1}=1$, reduce every exponent modulo $q-1$.  This gives a representative with exponents in $\{0,\ldots,q-2\}$ and total degree at most $q-1$.  Lemma~\ref{lem:torus-degree} therefore yields
\[
 \srk(\tau_\pi^{\boxt n})
 \le rM_{q-1}\left(n,\left\lfloor\frac{(q-1)n}{r}\right\rfloor\right).
\]
There are $\StirlingTwo{p}{r}$ set partitions with $r$ blocks.  Summing over $r\ge3$ proves the first inequality.

For the second inequality, use again
\[
 G_{q-1}(\rho)=1+\rho+\cdots+\rho^{q-2}.
\]
For $r\ge3$, put $D_r=\lfloor(q-1)n/r\rfloor$.  Since
$D_r\le(q-1)n/3$, the same generating-function argument as in Corollary~\ref{cor:toric-strong-e1} gives, for every $0<\rho<1$,
\[
 M_{q-1}(n,D_r)
 \le G_{q-1}(\rho)^n\rho^{-D_r}
 \le\left(G_{q-1}(\rho)\rho^{-(q-1)/3}\right)^n.
\]
Taking the infimum over $\rho$ yields
\[
 M_{q-1}(n,D_r)\le\vartheta_q^n
 \qquad(r\ge3).
\]
Multiplying by $r\StirlingTwo{p}{r}$ and summing gives the second inequality.

{For a sequence $S$, let $A$ be the set of distinct values appearing in $S$.  Every value occurs at most $p-1$ times, so $|A|\ge |S|/(p-1)$, and the claimed sequence bound follows from the set bound.}
\end{proof}

\subsection{Arithmetic progressions and cap sets on the torus}

When the characteristic is three, the equation
\[
 x+y+z=0
\]
is equivalent to the three-term arithmetic-progression equation $x+z=2y$.  A nonconstant solution necessarily has three distinct entries.  Thus distinctly $e_1$-free sets of length three are precisely three-term-progression-free sets, or cap sets in the underlying $\F_3$-vector space.  On the torus this gives a restricted-alphabet, coordinatewise punctured version of the cap-set problem.

The following are immediate numerical specializations of Theorem~\ref{thm:toric-zero-sum}, included to connect the restricted-alphabet theorem with the standard terminology of arithmetic progressions, cap sets and classical zero sums.

\begin{corollary}[Two classical restricted-alphabet consequences]\label{cor:toric-classical-consequences}
The following statements hold.
\begin{enumerate}[label=\textup{(\roman*)}]
\item If $A\subseteq(\F_9^\times)^n$ contains no nontrivial three-term arithmetic progression, equivalently if it is a cap set in the underlying $\F_3$-space $\F_9^n\cong\F_3^{2n}$, then
\[
 |A|\le3\vartheta_9^n,
 \qquad\vartheta_9\approx7.4810579074.
\]
{With $\gamma$ as in Proposition~\ref{prop:q3-sharp-base},
\[
 \vartheta_9<\gamma^2\approx7.5906014287.
\]}
Hence the toric estimate improves, at the exponential scale, the bound obtained by viewing $A$ as an ordinary cap set in $\F_3^{2n}$ and applying the Ellenberg--Gijswijt bound~\cite{EllenbergGijswijt}.

\item If $A\subseteq(\F_5^\times)^n$ contains no five pairwise distinct elements whose sum is zero, then
\[
 |A|\le120\vartheta_5^n,
 \qquad\vartheta_5\approx3.9556902435<2\sqrt5\approx4.4721359550.
\]
Thus the restricted-alphabet estimate improves, at the exponential scale, the bound inherited from the corresponding full-space theorem of Sauermann~\cite{SauermannDistinct}.
\end{enumerate}
\end{corollary}

\begin{proof}
For part~\textup{(i)}, apply Theorem~\ref{thm:toric-zero-sum} with $p=3$ and $q=9$; here $B_3=3$.  The interpretation in terms of arithmetic progressions follows from characteristic three.  The comparison with the ambient cap-set bound follows from $\F_9^n\cong\F_3^{2n}$ as additive groups and from the displayed numerical minimizations.

For part~\textup{(ii)}, apply Theorem~\ref{thm:toric-zero-sum} with $p=q=5$.  Since
\[
 B_5=3\StirlingTwo{5}{3}+4\StirlingTwo{5}{4}+5\StirlingTwo{5}{5}=120,
\]
the claimed bound follows.  The last comparison is numerical.
\end{proof}

\section{Elementary symmetric polynomials and higher-degree EGZ}\label{sec:higher-egz}

{We now transfer the restricted-alphabet linear theory to the higher-degree polynomial $e_{p-1}$.  Coordinatewise inversion identifies its toric zero set with the linear zero-sum equation; more generally, its contractions become weighted linear equations.  Support stratification then lifts the resulting estimates from $(\F_q^\times)^n$ to the whole space $\F_q^n$.}

\subsection{Inversion and the polynomial \texorpdfstring{$e_{p-1}$}{e(p-1)}}
{More generally, let $\pi=\{B_1,\ldots,B_r\}$ be an equality profile and put $\lambda_j=|B_j|$.  On the multiplicative torus, the corresponding contraction satisfies
\[
 (e_{p-1})_\pi(y_1,\ldots,y_r)
 =\left(\prod_{j=1}^r y_j^{\lambda_j}\right)
 \left(\sum_{j=1}^r\lambda_jy_j^{-1}\right).
\]
Thus its zero set is the weighted linear equation
\[
 \sum_{j=1}^r\lambda_jy_j^{-1}=0.
\]
This formula shows explicitly how the contractions appearing in Section~\ref{sec:partition-criteria} become weighted linear equations after inversion.}

\begin{corollary}[Toric strong and EGZ-type bounds for $e_{p-1}$]\label{cor:toric-epminus1}
Let $q=p^k>4$, with $p$ odd.

\begin{enumerate}[label=\textup{(\roman*)}]
\item If $A\subseteq(\F_q^\times)^n$ is strongly $e_{p-1}$-free of length $p$ relative to the torus, then
\[
 |A|\le
 pM_{q-1}\left(n,\left\lfloor\frac{(q-1)n}{p}\right\rfloor\right)
 \le p\eta_{q,p}^n,
 \qquad \eta_{q,p}<q-1.
\]
\item If $A\subseteq(\F_q^\times)^n$ is distinctly $e_{p-1}$-free of length $p$ relative to the torus, then
\[
 |A|\le B_p\vartheta_q^n,
 \qquad \vartheta_q<q-1.
\]
If $S$ is a sequence over $(\F_q^\times)^n$ with no length-$p$ subsequence on which $e_{p-1}$ vanishes, then
\[
 |S|\le(p-1)B_p\vartheta_q^n.
\]
\end{enumerate}
\end{corollary}

\begin{proof}
Coordinatewise inversion
\[
 \iota:(\F_q^\times)^n\longrightarrow(\F_q^\times)^n,
 \qquad
 \iota(x)_j=x_j^{-1},
\]
is a bijection, preserves pairwise distinctness, and maps constant tuples to constant tuples.  For nonzero scalar entries,
\[
 e_{p-1}(x_1,\ldots,x_p)
 =(x_1\cdots x_p)\left(x_1^{-1}+\cdots+x_p^{-1}\right).
\]
The product is nonzero, so a coordinatewise zero of $e_{p-1}$ is equivalent to a coordinatewise zero of $e_1$ formed by the inverses.  Hence inversion identifies strong $e_{p-1}$-freeness with strong $e_1$-freeness and distinct $e_{p-1}$-freeness with distinct $e_1$-freeness.  Part~\textup{(i)} follows from Corollary~\ref{cor:toric-strong-e1}, while part~\textup{(ii)} follows from Theorem~\ref{thm:toric-zero-sum}, and similarly for sequences.
\end{proof}

\begin{example}[Strong and EGZ-type $e_4$-avoidance on the $5$-ary torus]\label{ex:toric-e4-five}
For $p=q=5$,
\[
 \eta_{5,5}
 =\inf_{0<\rho<1}(1+\rho+\rho^2+\rho^3)\rho^{-4/5}
 \approx3.256361660<4.
\]
Hence every strongly $e_4$-free subset of $(\F_5^\times)^n$ of length five has size at most
\[
 5(3.256362)^n.
\]
For the distinct-variable, or toric higher-degree EGZ-type, variation, every subset containing no five pairwise distinct coordinatewise zeros of $e_4$ has size at most
\[
 120(3.955691)^n,
\]
and the corresponding sequence bound is four times this quantity.  Thus the same inversion identity gives nontrivial exponential bounds for both the strong and the distinct $e_4$-avoidance problems, measured against the torus alphabet size $4$.
\end{example}

\subsection{Support stratification and full-space bounds}
The toric bounds of Corollary~\ref{cor:toric-epminus1} can be transported to the whole space $\F_q^n$ by stratifying according to the support.  This mechanism was already used in \cite[Theorem~4.3]{CostaDellaFiore}, where the inverted stratum was treated inside the ambient additive group.  Here we retain the restricted alphabet: the inverted vectors lie in the multiplicative torus, and Theorem~\ref{thm:toric-zero-sum} supplies the sharper base $\vartheta_q$.  For $q=p=5$ this turns the trivial base in \cite[Theorem~4.3]{CostaDellaFiore} into a strict exponential improvement.

For $x\in\F_q^n$ write $\supp(x)=\{j\in[n] : x_j\neq 0\}$, and for
$P\subseteq[n]$ and $A\subseteq\F_q^n$ set
\[
A_P \;=\; \{\,x\in A \,:\, \supp(x)=P\,\}.
\]
Here, for a finite index set $P$, the notation $X^P$ denotes the set of functions from $P$ to $X$.  Restriction to the coordinates in $P$ identifies $A_P$ with a subset of the
torus $(\F_q^{\times})^{P}$; this identification is injective and maps
constant tuples to constant tuples and pairwise distinct tuples to pairwise
distinct tuples.

The stratification is compatible with every elementary symmetric polynomial.  This elementary bookkeeping step is the one already used in the proof of~\cite[Theorem~4.3]{CostaDellaFiore}; we reproduce it to isolate exactly where the new toric input enters.

\begin{lemma}[Support stratification]\label{lem:stratification}
Let $q$ be a prime power, let $t\geq m\geq 1$, let $A\subseteq\F_q^n$ and let
$P\subseteq[n]$. Elements $x^{(1)},\dots,x^{(t)}\in A_P$ satisfy
\[
e_m\bigl(x^{(1)},\dots,x^{(t)}\bigr)=0
\]
coordinatewise in $\F_q^n$ if and only if their restrictions to the
coordinates in $P$ satisfy the same system coordinatewise in
$(\F_q^{\times})^{P}$. Consequently:
\begin{itemize}
\item[(i)] if $A$ is strongly $e_m$-free of length $t$, then each $A_P$,
viewed inside $(\F_q^{\times})^{P}$, is strongly $e_m$-free of length $t$
relative to the torus;
\item[(ii)] if $A$ contains no $t$ pairwise distinct elements annihilating
$e_m$ coordinatewise, then each $A_P$ is distinctly $e_m$-free of length $t$
relative to the torus.
\end{itemize}
\end{lemma}

\begin{proof}
All the elements of $A_P$ vanish exactly on the coordinates outside $P$.
Hence, for $j\notin P$, the $j$-th coordinate of the tuple is
$(0,\dots,0)$ and $e_m(0,\dots,0)=0$ holds automatically, since $m\geq 1$.
For $j\in P$ all the entries are nonzero. Therefore the coordinatewise system
over $\F_q^n$ is equivalent to the coordinatewise system over the torus
$(\F_q^{\times})^{P}$ for the restricted tuples. Claims (i) and (ii) follow
because restriction to $P$ is injective on $A_P$ and preserves constancy and
pairwise distinctness of tuples.
\end{proof}

Combining Lemma~\ref{lem:stratification} with Corollary~\ref{cor:toric-epminus1} and
summing over the strata yields the following full-space bounds. Recall that
\[
\vartheta_q=\inf_{0<\rho<1}\,(1+\rho+\dots+\rho^{q-2})\,\rho^{-(q-1)/3},
\qquad
\eta_{q,p}=\inf_{0<\rho<1}\,(1+\rho+\dots+\rho^{q-2})\,\rho^{-(q-1)/p},
\]
that $\vartheta_q,\eta_{q,p}<q-1$ for $q>4$, and that
$B_p=\sum_{r=3}^{p} r\StirlingTwo{p}{r}$.

\begin{theorem}[Full-space bounds for $e_{p-1}$]\label{thm:fullspace}
Let $q=p^k>4$, with $p$ an odd prime.
\begin{itemize}
\item[(i)] If $A\subseteq\F_q^n$ is strongly $e_{p-1}$-free of length $p$,
then
\[
|A| \;\leq\; p\sum_{w=0}^{n}\binom{n}{w}
M_{q-1}\!\left(w,\left\lfloor\frac{(q-1)w}{p}\right\rfloor\right)
\;\leq\; p\,\bigl(1+\eta_{q,p}\bigr)^{n},
\qquad 1+\eta_{q,p}<q.
\]
\item[(ii)] If $A\subseteq\F_q^n$ contains no $p$ pairwise distinct elements
annihilating $e_{p-1}$ coordinatewise, then
\[
|A| \;\leq\; \sum_{w=0}^{n}\binom{n}{w}\sum_{r=3}^{p} r\StirlingTwo{p}{r}
M_{q-1}\!\left(w,\left\lfloor\frac{(q-1)w}{r}\right\rfloor\right)
\;\leq\; B_p\,\bigl(1+\vartheta_q\bigr)^{n},
\qquad 1+\vartheta_q<q.
\]
\item[(iii)] Consequently,
\[
\EGZ\bigl(p,\F_q^n,p-1\bigr) \;\leq\; (p-1)\,B_p\,\bigl(1+\vartheta_q\bigr)^{n}+1.
\]
\end{itemize}
\end{theorem}

\begin{proof}
(i) By Lemma~\ref{lem:stratification}(i), for every $P\subseteq[n]$ the
stratum $A_P$, viewed inside the torus of dimension $|P|$, is strongly
$e_{p-1}$-free of length $p$ relative to the torus. By Corollary~\ref{cor:toric-epminus1}(i),
\[
|A_P| \;\leq\; p\, M_{q-1}\!\left(|P|,
\left\lfloor\frac{(q-1)|P|}{p}\right\rfloor\right)
\;\leq\; p\,\eta_{q,p}^{\,|P|}.
\]
Summing over all $P\subseteq[n]$ and grouping by $w=|P|$ gives the first
inequality; the second follows from the binomial theorem, since
$\sum_{w}\binom{n}{w}\eta_{q,p}^{\,w}=(1+\eta_{q,p})^{n}$. Finally
$1+\eta_{q,p}<q$ because $\eta_{q,p}<q-1$.

(ii) Identical, using Lemma~\ref{lem:stratification}(ii) and the first bound in Theorem~\ref{thm:toric-zero-sum}, transported through coordinatewise inversion as in Corollary~\ref{cor:toric-epminus1}(ii), on each stratum, together with the estimate
$M_{q-1}(w,\lfloor (q-1)w/r\rfloor)\leq
M_{q-1}(w,\lfloor (q-1)w/3\rfloor)\leq\vartheta_q^{\,w}$ for every $r\geq 3$.

(iii) Let $S$ be a sequence over $\F_q^n$ with no subsequence of length $p$
annihilating $e_{p-1}$. Since
\[
e_{p-1}(a,\dots,a)=\binom{p}{p-1}a^{p-1}=p\,a^{p-1}=0
\]
in characteristic $p$ for every $a\in\F_q^n$.  {Let $A$ be the set of distinct values appearing in $S$.  Since no value occurs $p$ times, $|A|\geq |S|/(p-1)$.} If $A$
contained $p$ pairwise distinct elements annihilating $e_{p-1}$, choosing one
occurrence of each in $S$ would give a forbidden subsequence. Hence $A$
satisfies the hypothesis of (ii), and $|S|\leq (p-1)B_p(1+\vartheta_q)^n$.  {Thus every sequence with length greater than
\[
 (p-1)B_p(1+\vartheta_q)^n
\]
contains the required subsequence.  Consequently,
\[
 \EGZ(p,\F_q^n,p-1)
 \le \left\lfloor (p-1)B_p(1+\vartheta_q)^n\right\rfloor+1
 \le (p-1)B_p(1+\vartheta_q)^n+1.
\]
This proves~\textup{(iii)}.}
\end{proof}

For $q=p=5$ this turns the trivial exponential base furnished by
\cite[Theorem~4.3]{CostaDellaFiore} into a strict full-space bound.

\begin{corollary}[The case $q=p=5$]\label{cor:q5}
Every subset of $\F_5^n$ containing no five pairwise distinct elements
annihilating $e_4$ coordinatewise has cardinality at most
$120\,(1+\vartheta_5)^n$, and
\[
\EGZ\bigl(5,\F_5^n,4\bigr)\;\leq\; 480\,\bigl(1+\vartheta_5\bigr)^{n}+1,
\qquad 1+\vartheta_5\approx 4.9556902 \,<\, 5.
\]
Moreover, every strongly $e_4$-free subset of $\F_5^n$ of length five has
cardinality at most
\[
5\,\bigl(1+\eta_{5,5}\bigr)^{n},
\qquad 1+\eta_{5,5}\approx 4.2563617 \,<\, 5.
\]
\end{corollary}

\begin{remark}
The strong bound in Theorem~\ref{thm:fullspace}(i) and
Corollary~\ref{cor:q5} lies outside the range of Corollary~\ref{thm:strong-em},
since $t=p\leq 2(p-1)=2m$: no nontrivial exponential bound for strongly $e_{p-1}$-free subsets of
the full space $\F_q^n$ was available from the degree method.
\end{remark}

\begin{remark}\label{rem:comparison-previous}
Theorem~\ref{thm:fullspace}(iii) should be compared with
\cite[Theorem~4.3]{CostaDellaFiore}, which gives
\[
\EGZ(p,\F_q^n,p-1)\leq (u_{p,k}+1)^{n+o(n)},
\]
where $u_{3,k}=2.756^k$ for $k\geq2$ and
$u_{p,k}=(2\sqrt p)^k$ for $p\neq3$, $p^k\geq7$, and also provides sharper computational bounds for $q=3^k$, $2\leq k\leq5$, by optimizing the entropy of the exact coefficient support.

Both approaches use support stratification and coordinatewise inversion. In
\cite{CostaDellaFiore}, however, each inverted stratum is treated inside the ambient additive group, whereas here the restricted alphabet is retained and Theorem~\ref{thm:toric-zero-sum} is applied on the torus. This gives a bound uniform in the size of the stratum, without a cutoff or an $o(n)$-loss, but it is not uniformly sharper. For $q=9$, the resulting full-space base
\[
1+\vartheta_9\approx8.4810579
\]
is larger than the rounded computational base $8.315$ from
\cite{CostaDellaFiore}.  As noted in Remark~\ref{rem:q3-dual-certificate}, the separate direct slice-rank analysis of Section~\ref{sec:q3-gap} gives $C_2\approx8.311438619$, agreeing with the underlying field-specific numerical optimization to the precision reported there; the weaker value $1+\vartheta_9$ belongs specifically to the toric-stratification route discussed in this remark.  In the full-space setting, the main qualitative gain occurs for $q=p=5$, where the present method yields a nontrivial exponential bound while the previous general closed-form estimate was trivial.  For $q=p=3$, the toric reduction gives only the trivial base $3$, and, to the best of our knowledge, it remains open whether there exists $c<3$ such that
\[
 \EGZ(3,\F_3^n,2)\le c^{n+o(n)}.
\]
\end{remark}

\section*{Declaration of generative AI assistance}

During the preparation of this work, the authors used OpenAI's ChatGPT and Anthropic's Claude as interactive aids in exploring possible approaches to mathematical and expository questions arising during the development of the paper.  All AI-assisted suggestions were critically assessed, and all arguments and statements appearing in the manuscript were independently verified by the authors, who take full responsibility for its content.


\begin{thebibliography}{99}


\bibitem{CaroSchmitt}
Y.~Caro and J.~R.~Schmitt,
\emph{Higher degree Erd\H{o}s--Ginzburg--Ziv constants},
Integers \textbf{22} (2022), Paper No.~A102, 17 pp.


\bibitem{CoverThomas}
T.~M.~Cover and J.~A.~Thomas,
\emph{Elements of Information Theory}, second ed.,
Wiley-Interscience, Hoboken, NJ, 2006;
\href{https://doi.org/10.1002/047174882X}{doi:10.1002/047174882X}.


\bibitem{CostaDellaFiore}
S.~Costa and S.~Della Fiore,
\emph{Bounds on the higher degree Erd\H{o}s--Ginzburg--Ziv constants over $\F_q^n$},
Arch. Math. (Basel) \textbf{122} (2024), 17--29;
\href{https://doi.org/10.1007/s00013-023-01916-4}{doi:10.1007/s00013-023-01916-4};
arXiv:2211.03682v2, containing the qualitative slice-rank result recalled in Proposition~\ref{prop:q3-support-gap}.

\bibitem{CrootLevPach}
E.~Croot, V.~F.~Lev, and P.~Pach,
\emph{Progression-free sets in $\mathbb Z_4^n$ are exponentially small},
Ann. of Math. (2) \textbf{185} (2017), 331--337.

\bibitem{EllenbergGijswijt}
J.~S.~Ellenberg and D.~Gijswijt,
\emph{On large subsets of $\F_q^n$ with no three-term arithmetic progression},
Ann. of Math. (2) \textbf{185} (2017), 339--343.

\bibitem{FoxSauermann}
J.~Fox and L.~Sauermann,
\emph{Erd\H{o}s--Ginzburg--Ziv constants by avoiding three-term arithmetic progressions},
Electron. J. Combin. \textbf{25} (2018), Paper No.~P2.14;
\href{https://doi.org/10.37236/7275}{doi:10.37236/7275}.


\bibitem{Lucas1878}
\'E.~Lucas,
\emph{Sur les congruences des nombres eul\'eriens et des coefficients diff\'erentiels des fonctions trigonom\'etriques suivant un module premier},
Bull. Soc. Math. France \textbf{6} (1878), 49--54;
\href{https://archive.numdam.org/item/BSMF_1878__6__49_1/}{Numdam}.

\bibitem{NaslundEGZ}
E.~Naslund,
\emph{Exponential bounds for the Erd\H{o}s--Ginzburg--Ziv constant},
J. Combin. Theory Ser. A \textbf{174} (2020), 105185;
\href{https://doi.org/10.1016/j.jcta.2019.105185}{doi:10.1016/j.jcta.2019.105185}.

\bibitem{NaslundPartition}
E.~Naslund,
\emph{The partition rank of a tensor and $k$-right corners in $\F_q^n$},
J. Combin. Theory Ser. A \textbf{174} (2020), 105190;
\href{https://doi.org/10.1016/j.jcta.2019.105190}{doi:10.1016/j.jcta.2019.105190}.


\bibitem{OmarPartition}
M.~Omar,
\emph{Partition rank and partition lattices},
Order \textbf{42} (2025), 371--388;
\href{https://doi.org/10.1007/s11083-024-09685-7}{doi:10.1007/s11083-024-09685-7}.



\bibitem{SauermannDistinct}
L.~Sauermann,
\emph{On the size of subsets of $\F_p^n$ without $p$ distinct elements summing to zero},
Israel J. Math. \textbf{243} (2021), 63--79.

\bibitem{TaoBlog}
T.~Tao,
\emph{A symmetric formulation of the Croot--Lev--Pach--Ellenberg--Gijswijt capset bound},
What's new, 18 May 2016;
\href{https://terrytao.wordpress.com/2016/05/18/a-symmetric-formulation-of-the-croot-lev-pach-ellenberg-gijswijt-capset-bound/}{blog post}.

\bibitem{TaoSawin}
T.~Tao and W.~Sawin,
\emph{Notes on the slice rank of tensors},
What's new, 24 August 2016;
\href{https://terrytao.wordpress.com/2016/08/24/notes-on-the-slice-rank-of-tensors/}{joint blog post}.

\end{thebibliography}
\end{document}